\documentclass[letterpaper, 10 pt, conference]{ieeeconf}  % Comment this line out if you need a4paper

\usepackage[cmex10]{amsmath}
\usepackage[pdftex]{graphicx}
\usepackage{epsfig}
\usepackage{amssymb}
\usepackage{color}
\usepackage{graphics}
\usepackage{empheq}
\usepackage{hyperref}
\usepackage{array}
\usepackage{cite}
\usepackage{subfig}
\usepackage{bm}

\newtheorem{assum}{Assumption}

\newtheorem{prop}{Proposition}

\newtheorem{remark}{Remark}
\newtheorem{lem}{Lemma}

\def\f{\frac}

\def\i1n{i=1,\cdots,n}
\def\j1n{j=1,\cdots,n}
\def\ij1n{i,j=1,\cdots,n}

\newcommand{\be}{\begin{equation}}
\newcommand{\ee}{\end{equation}}
\newcommand{\beq}{\begin{equation*}}
\newcommand{\eeq}{\end{equation*}}
\newtheorem{thm}{Theorem}

\IEEEoverridecommandlockouts                              % This command is only needed if 
\title{\LARGE \bf
Output Feedback Control of the One-Phase Stefan Problem
}

\author{Shumon Koga, Mamadou Diagne, and Miroslav Krstic% <-this % stops a space
\thanks{Shumon Koga and Miroslav Krstic are with the Department of Mechanical and Aerospace Engineering, U.C. San Diego, 9500 Gilman Drive, La Jolla, CA, 92093-0411, {\tt\small skoga@ucsd.edu} and {\tt\small krstic@ucsd.edu}}
\thanks{Mamadou Diagne is with the Department of Mechanical Engineering of the University of Michigan Ann Arbor, MI 48109-2102, USA,
{\tt\small mdiagne@umich.edu}}
}

\begin{document}

\maketitle
\thispagestyle{empty}
\pagestyle{empty}

%%%%%%%%%%%%%%%%%%%%%%%%%%%%%%%%%%%%%%%%%%%%%%%%%%%%%%%%%%%%%%%%%%%%%%%%%%%%%%%

\begin{abstract}
In this paper, a backstepping observer and an output feedback control law are designed for   the stabilization of the one-phase Stefan problem. The present result is  an improvement of the recent  full state feedback backstepping controller proposed  in our previous contribution. The one-phase Stefan problem describes the time-evolution of a temperature profile in a liquid-solid material and its liquid-solid moving interface. This phase transition problem is mathematically  formulated as a 1-D diffusion Partial Differential Equation (PDE) of the melting zone defined on a time-varying spatial domain described by an Ordinary Differential Equation (ODE). We propose a  backstepping observer allowing to estimate the temperature profile along the melting zone based on the available  measurement, namely, the solid phase length. The designed observer and the output feedback controller ensure the exponential stability of the estimation errors, the moving interface, and the ${\cal H}_1$-norm of the distributed temperature while keeping physical constraints, which is shown with the restriction on the gain parameter of the observer and the setpoint.
\end{abstract}

%%%%%%%%%%%%%%%%%%%%%%%%%%%%%%%%%%%%%%%%%%%%%%%%%%%%%%%%%%%%%%%%%%%%%%%%%%%%%%%%
\section{INTRODUCTION}
Liquid-solid phase transition appear in various kinds of science and engineering processes. Typical applications include sea-ice melting and freezing\cite{wett91}, continuous casting of steel \cite{petrus2012}, crystal-growth \cite{conrad_90}, and thermal energy storage system\cite{Belen03}. The physical description of these processes is that, a temperature profile in a liquid-solid material promotes the dynamics of a liquid-solid interface due to the phase transition induced by melting or solidification processes. A mathematical model of such a physical process is called the Stefan problem\cite{Gupta03}, which is formulated by a diffusion PDE defined on a time-varying spatial domain. The domain's dynamics is described by an ODE actuated by the Neumann boundary value of the PDE state. 
 
For control objectives, infinite-dimensional frameworks that lead to significant mathematical complexities in the process characterization have been developed for the stabilization of  the temperature profile and the moving interface of a 1D Stefan problem. For instance,\cite{petrus2012} proposed  an enthalpy-based feedback  to  ensure asymptotical stability of the temperature profile and the moving boundary at a desired reference. In \cite{maidi2014} a  geometric control approach which enables  to adjust the position of a liquid-solid interface at a desired setpoint, and the exponential stability of the ${\cal L}_2$-norm of the distributed temperature is developed using a Lyapunov analysis. It is worth to mention that \cite{maidi2014}  considers  \emph{a priori} strictly positive  boundary input   and a moving boundary  which is assumed to be  a  non-decreasing function of time.

In this paper, a backstepping observer \cite{krstic2008boundary,andrew2004} and an output feedback control law are developed for the stabilization of the interface position and the temperature profile of the melting zone of  the one-phase Stefan problem  \cite{petrus2012, Gupta03, maidi2014}. The present work is an improvement of state-feedback result which was proposed in our previous contribution \cite{Shumon16}. While \cite{Izadi15} designed an output feedback controller that ensures the exponential stability of an unstable parabolic PDE system coupled with a priori-known moving interface, the observer design for the Stefan problem is rarely addressed in the literature. We propose  a backstepping transformation which enables to deal with the  one-phase Stefan problem in which the dynamics of the moving boundary is not explicitly given due to its state-dependency. Such a transformation is exploited to design an estimator and a controller standing as  an extension of the one proposed in \cite{krstic2009, Gian2011} for coupled  linear PDE-ODEs systems  defined on a fixed spatial domain. Our designed observer and the output feedback controller achieve the exponential stabilization of the estimation error, the temperature profile, and the moving interface to the desired  references  in the ${\cal H}_{1}$-norm under the restrictions on the observer gain and the setpoint.

This paper is organized as follows: The one-phase Stefan problem  is presented  in Section \ref{model}, and a brief review of full-state feedback result \cite{Shumon16} is stated in Section \ref{state-feedback}. Section \ref{sec:estimation} explains the observer design and the output feedback control problem with the statement of the main results. Section \ref{nonlineartarget} introduces a backstepping transformation for moving boundary problems which allows to design the observer gains and the output feedback control law. The physical constraints of this problem are stated and guaranteed by restricting the setpoint and the gain parameter of the observer in Section \ref{sec:constraints}. The Lyapunov stability of the closed-loop system is established in Section \ref{stability}.  Supportive numerical simulations are provided in Section \ref{simulation}. The paper ends with final remarks and future directions discussed in  Section \ref{conclusion}.

\section{Description of the Physical Process}\label{model}
\begin{figure}[htb]
\centering
\includegraphics[width=3.2in]{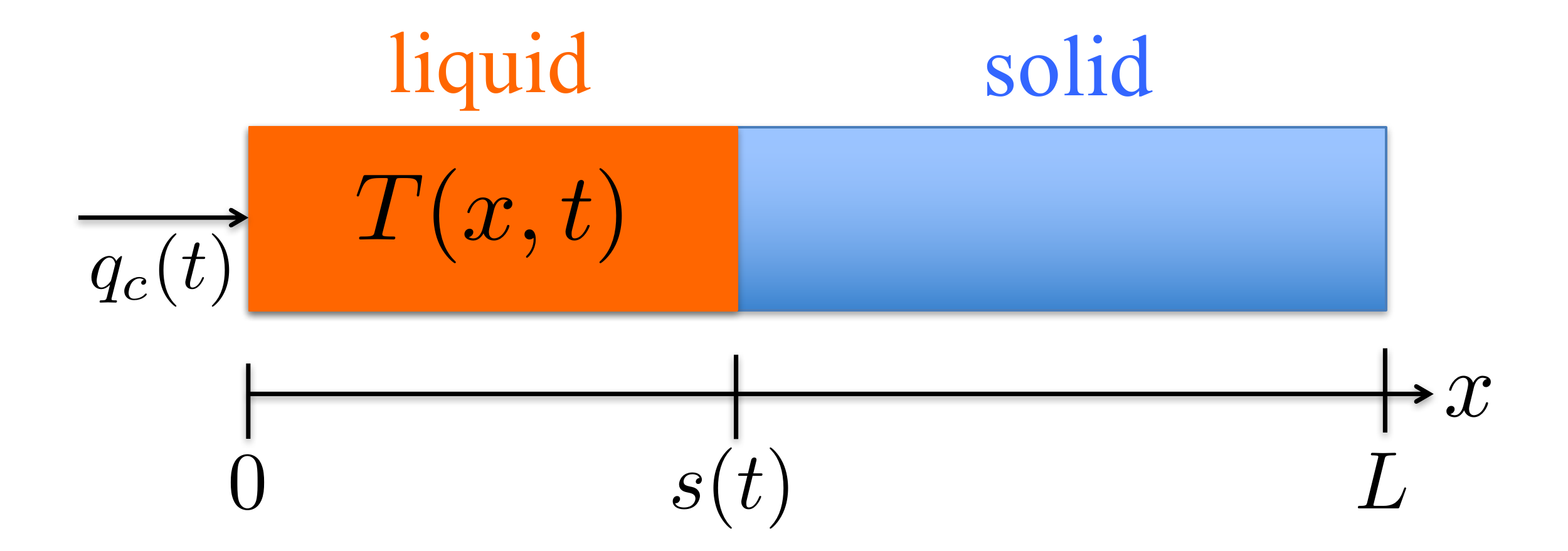}\\
\caption{Schematic of 1D Stefan problem.}
\label{fig:stefan}
\end{figure}
Consider a physical model which describes the melting or the solidification mechanism in a pure one-component material of length $L$ in one dimension. In order to describe the position at which phase transition from liquid to solid occurs (or equivalently, in the reverse direction) mathematically, we divide the domain $[0, L]$ into the two time-varying sub-domains, namely,  $[0,s(t)]$ occupied by the liquid phase, and $[s(t),L]$ by the solid  phase.  A heat flux is entering the system at  the boundary at $x=0$ of the liquid phase, which affects the dynamics of the liquid-solid interface. Assuming that the temperature in the liquid phase is not lower than the melting temperature of the material, we derive the following system described by:
\begin{itemize}
\item the diffusion equation of the temperature in the liquid phase which is written as
\begin{align}\label{eq:stefanPDE}
T_t(x,t)=\alpha T_{xx}(x,t), \hspace{1mm} 0\leq x\leq s(t), \hspace{1mm} \alpha :=\f{k}{\rho C_p}, 
\end{align}
with the boundary conditions
\begin{align}\label{eq:stefancontrol}
-k T_x(0,t)=&q_c(t), \\ \label{eq:stefanBC}
T(s(t),t)=&T_m,
\end{align}
and the initial values
\begin{align}\label{eq:stefanIC}
T(x,0)=T_0(x), \quad s(0) = s_0, 
\end{align}
where $T(x,t)$,  $q_c(t)$,  $\rho$, $C_p$ and  $k$ are the distributed temperature of the liquid phase, manipulated heat flux, liquid density, the liquid heat capacity, and the liquid heat conductivity, respectively.
\end{itemize}
\begin{itemize}
\item the local  energy balance at the liquid-solid interface $x=s(t)$  which yields to  the following ODE
\begin{align}\label{eq:stefanODE}
 \dot{s}(t)=-\beta T_x(s(t),t), \quad \beta :=\frac{k}{\rho \Delta H^*}
\end{align}
that describes the dynamics of moving boundary where $\Delta H^*$ denotes the latent heat of fusion.
\end{itemize}
For the sake of brevity, we refer the readers to  \cite{Gupta03}, where the Stefan condition in the case of a solidification process is derived. 
\begin{remark}As the moving interface  $s(t)$ is not explicitly given, the problem defined in  \eqref{eq:stefanPDE}--\eqref{eq:stefanODE}  is a highly nonlinear
problem. \end{remark}
\begin{remark}\label{assumption}Due to the so-called isothermal interface condition
 that prescribes the melting temperature $T_m$ at the interface through  \eqref{eq:stefanBC},
this form of the Stefan problem is a reasonable model only if the following conditions hold:
\begin{align}
T(x,t) \geq &T_m, \quad \textrm{ for all }\quad x \in [0,s(t)],\\
\label{moving}\dot{s}(t)\geq&0, \quad \textrm{ for all }\quad t\geq0. 
\end{align}
\end{remark}
From  Remark \ref{assumption}, it is plausible to assume $s_0>0$ and the existence of a positive constant $H>0$ such that
\begin{align}\label{eq:stefanICbound}
0\leq T_0(x)-T_m\leq H(s_0-x). 
\end{align}
We recall the following lemma that ensures the validity of the model \eqref{eq:stefanPDE}--\eqref{eq:stefanODE}.
\begin{lem}\label{lemma1}
For any $q_c(t)>0$ on the finite time interval $(0,\bar{t})$, $T(x,t)>T_m, ~\forall x\in(0,s(t))$ and $\forall t\in(0,\bar{t})$. And then $\dot{s}(t) >0$, $\forall t\in(0,\bar{t})$. 
\end{lem}
The proof of Lemma \ref{lemma1} is provided by maximum principle and Hopf's Lemma as shown in \cite{Gupta03}, \cite{nonlinearPDE}.

\section{State Feedback Control }\label{state-feedback}
In this section, we recall the main result of the backstepping state-feedback control of the 1D Stefan problem  \cite{Shumon16}. The control objective is to drive the moving boundary $s(t)$ to a reference setpoint $s_r$ by manipulating the heat controller $q_c(t)$. From a physical point of view, for a positive heat controller $q_c(t)$,  the irreversibility of the process restrict a priori the choice of the desired setpoint $s_r$. The control objective can be achieved  only if such constraints are satisfied. The following assumption is stated to satisfy the aforementioned physical constraints. 
\begin{assum}\label{assum}
The setpoint $s_r$ is chosen to satisfy the following inequality
\begin{align}\label{compatibility}
s_r>s_0+\frac{\beta}{\alpha }\int_{0}^{s_0} (T_0(x)-T_m) dx. 
\end{align}
\end{assum}
Assumption \ref{assum} is necessary to achieve the control objective due to the energy conservation law given by
\begin{align}\label{1stlaw}
\frac{d}{dt}\left(\frac{1}{\alpha}\int_0^{s(t)} (T(x,t)-T_m) dx +\frac{1}{\beta} s(t)\right) = \frac{q_c(t)}{k}. 
\end{align} 
The left hand side of \eqref{1stlaw} denotes the growth of internal energy. With a positive heat control $q_{c}(t)>0$, the internal energy for a given setpoint must be greater than the initial internal energy, which leads to the condition \eqref{compatibility}. 

Suppose that both $T(x,t)$ and $s(t)$ are measured  $\forall x\in [0, s(t)]$ and $\forall t\geq 0$. Then, the following theorem holds:  
\begin{thm}\label{Theo-1}
Consider a closed-loop system consisting of the plant \eqref{eq:stefanPDE}--\eqref{eq:stefanODE} and the control law
\begin{align}\label{Fullcontrol}
q_c(t)=-ck \left(\frac{1}{\alpha}\int_0^{s(t)} (T(x,t)-T_m) dx +\frac{1 }{\beta}(s(t)-s_r)\right). 
\end{align}
where $c>0$ is an arbitrary controller gain. Assume that the initial values $(T_0(x), s_0)$ are compatible with the control law and satisfies \eqref{eq:stefanICbound}. Then, for any reference setpoint $s_r$ satisfying \eqref{compatibility}, the closed-loop system is exponentially stable in the sense of the norm
\begin{align}
||T-T_m||_{{\cal H}_1}^2+(s(t)-s_r)^2.
\end{align}
\end{thm}
\begin{proof}
The control law \eqref{Fullcontrol} was derived using the following backstepping  transformation
\begin{align}
u(x,t)= &T(x,t)-T_m,  \quad X(t)=s(t)-s_r, \\
\label{eq:DBST}
w(x,t)=&u(x,t)-\frac{c}{\alpha} \int_{x}^{s(t)} (x-y)u(y,t) dy\nonumber\\
&+\frac{c}{\beta}(s(t)-x) X(t), 
\end{align}
introduced in  \cite{Shumon16} with the aim to transform the system \eqref{eq:stefanPDE}--\eqref{eq:stefanODE} into a target system. Noting that $q_c(t)>0$ is required by Remark \ref{assumption} and Lemma \ref{lemma1} to remain the model validity, the overshoot beyond the reference $s_r$ is prohibited to achieve the control objective $s(t) \to s_r$ due to its irreversible process \eqref{moving}, which means $s(t)<s_r$ is required to be satisfied for $\forall t>0$. These two conditions 
\begin{align}\label{physical constraints}
q_c(t)>0, \quad  s(t)<s_r, \quad  \forall t>0,
\end{align}
namely  the "physical constraints", are satisfied under the setpoint restriction \eqref{compatibility}. With the help of \eqref{moving} and \eqref{physical constraints}, the target system was shown to be exponentially stable.  The detailed  proof of Theorem \ref{Theo-1} is established in \cite{Shumon16}.
\end{proof}
\section{Control and Estimation Problem Statement and Main Results}\label{sec:estimation}
\subsection{ Problem Statement }\label{statement}
In  \cite{Shumon16}, the authors considered the full-state feedback control problem, in which the controller implementation requires available   measurements of the  temperature profile $T(x,t)$ along the domain $[0, s(t)]$ and the moving interface position $s(t)$. Under these conditions, the practical relevance of the proposed solution is relatively limited. In the present work, we extended the full-state feedback results considering  moving interface position $s(t)$ as the only available measurement.
\subsection{Observer Design}\label{Estimation}
 Suppose that the interface position is obtained as the only available measurement $Y(t)=s(t)$. Then, denoting the estimates of the temperature $\hat T(x,t)$, the following theorem holds:
\begin{thm}\label {observer}
Consider the following  closed-loop system  of the observer 
\begin{align}
 \label{observerPDE}\hat{T}_t(x,t)=&\alpha \hat{T}_{xx}(x,t) \nonumber\\
 & - P_1(x,s(t))\left(\frac{\dot{Y}(t)}{\beta} + \hat{T}_x(s(t),t)\right),  \\
 \label{observerBC2}-k\hat{T}_x(0,t)=&q_c(t), \\
\label{observerBC1}\hat{T}(s(t),t)=&T_m, 
\end{align}
where $x\in[0,s(t)]$, and the observer gain $P_1(x,s(t))$ is 
\begin{align}\label{P1x}
P_1(x,s(t)) =  -\lambda s(t)\frac{I_1\left(\sqrt{\frac{\lambda}{\alpha}\left(s(t)^2-x^2\right)}\right)}{\sqrt{\frac{\lambda}{\alpha} \left(s(t)^2-x^2\right)}}
\end{align}
with an observer gain $\lambda>0 $. Assume that the two physical constraints \eqref{physical constraints} are satisfied. Then, for all  $\lambda >0$, the observer error system is exponentially stable in the sense of the norm 
\begin{align}
||T-\hat{T}||_{{\cal H}_1}^2 . 
\end{align}
\end{thm}

\subsection{Output Feedback Control }
The design of the output feedback controller is achieved using the reconstruction of the state through the exponentially convergent observer defined in Theorem \ref{observer} and based on the only available messurement $Y(t)$. We propose the following theorem:

\begin{thm}\label{outputthm}
Consider the closed-loop system  \eqref{eq:stefanPDE}--\eqref{eq:stefanODE} with the measurement $Y(t)=s(t)$ and the observer \eqref{observerPDE}-\eqref{observerBC1} and the output feedback control law
\begin{align}\label{outctr}
q_c(t)=&-ck\left(\frac{1}{\alpha} \int_{0}^{Y(t)} \left(\hat{T}(x,t)-T_m\right) dx\right.\notag\\
&\left.+\frac{1 }{\beta} \left(Y(t)-s_r\right) \right). 
\end{align}
Assume that the initial values $\left(\hat{T}_{0}(x), s_0\right)$ are compatible with the control law and the initial plant states $(T_{0}(x), s_0)$ satisfy \eqref{eq:stefanICbound}. Additionally, assume that the upper bound of the initial temperature is known, i.e. the Lipschitz constant $H$ in \eqref{eq:stefanICbound} is known. Then, by setting an initial temperature estimation $\hat{T}_0(x)$, gain parameter of the observer $\lambda$, and the setpoint $s_r$ to satisfy 
\begin{align}\label{restriction1}
\hat{T}_0(x) =&T_m +  \hat{H}(s_0 - x), \quad \\
\label{restriction2} \lambda <& \frac{4\alpha}{s_0^2} \left(1-\frac{H}{\hat{H}}\right), \\
\label{restriction3} s_r>&s_0+\frac{\beta s_0^2}{2\alpha } \hat{H}, 
\end{align}
with a choice of a parameter $\hat{H}>H$, the closed-loop system is exponentially stable in the sense of the norm 
\begin{align}
&||T-\hat{T}||_{{\cal H}_1}^2+||T-T_m||_{{\cal H}_1}^2+(s(t)-s_r)^2 .
\end{align}
\end{thm}

The design procedure is presented in Section \ref{nonlineartarget} and \ref{sec:constraints}, and the proofs of Theorem \ref{observer}  and Theorem \ref{outputthm} are provided in Section \ref{stability}.

\section{Backstepping Transformation for Moving Boundary Formulation}\label{nonlineartarget}
We recall that  the reference error of liquid temperature and the moving interface are  denoted as $u(x,t)=T(x,t)-T_m$ and $X(t)=s(t)-s_r$, respectively (see Section \ref{state-feedback}). Defining the controller and the  available measurement as  $U(t) = -q_c(t)/k$ and $Y(t)=s(t)$, respectively,  the coupled system \eqref{eq:stefanPDE}--\eqref{eq:stefanODE} are written as
\begin{align}\label{eq:stefanPDE1}
u_{t}(x,t) =& \alpha u_{xx}(x,t), \quad 0\leq x\leq s(t)\\
u_x(0,t) =& U(t),\quad \label{eq:stefanBC1}u(s(t),t) =0,\\
\label{eq:stefanODE1}\dot{X}(t) =&-\beta u_x(s(t),t).
\end{align}
For the reference error  system, namely, the $u$-system  \eqref{eq:stefanPDE1}--\eqref{eq:stefanODE1},  we  consider the following observer:
\begin{align}
\hat{u}_t(x,t)=&\alpha  \hat{u}_{xx}(x,t)\nonumber\\&+P_1(x,s(t))\left(-\frac{1}{\beta}\dot{Y}(t)-\hat{u}_x(s(t),t)\right), \label{obsfirst}\\
\hat{u}(s(t),t)=&0, \\
\label{obslast}\hat{u}_x(0,t)=&U(t), 
\end{align}
where $P_1(x,s(t))$ is the observer gain that needs to be determined. Defining   error variable  of $u$-system as $\tilde{u}(x,t)=u(x,t)-\hat{u}(x,t)$ and combining \eqref{eq:stefanPDE1}--\eqref{eq:stefanODE1} with \eqref{obsfirst}--\eqref{obslast},  the $\tilde{u}$-system is written as
\begin{align}\label {eq:errorPDE}
\tilde{u}_t(x,t)=& \alpha \tilde{u}_{xx}(x,t)-P_1(x,s(t))\tilde{u}_x(s(t),t),\\
\label{eq:errorBC2}\tilde{u}(s(t),t)=&0, \quad \tilde{u}_x(0,t)=0.
\end{align}

\subsection{Observer Target System}
\subsubsection{Direct transformation}
In this section, we introduce a moving boundary backstepping transformation and observer gains motivated by a fixed boundary backstepping transformation \cite{krstic2009, Gian2011} as
\begin{align}\label{eq:errorDBST}
\tilde{u}(x,t)=&\tilde{w}(x,t)+\int_x^{s(t)} P(x,y)\tilde{w}(y,t) dy, 
\end{align}
which transforms the $\tilde{u}$-system in \eqref{eq:errorPDE}-\eqref{eq:errorBC2} into the following exponentially stable target system 
\begin{align}\label{wtilde1}
\tilde{w}_t(x,t)=& \alpha \tilde{w}_{xx}(x,t)- \lambda \tilde{w}(x,t),\\
\tilde{w}(s(t),t)=&0, \quad \tilde{w}_x (0,t)=0.\label{wtilde3}
\end{align}
Taking the derivative of \eqref {eq:errorDBST} with respect to $t$ and $x$ along the solution of \eqref{wtilde1}-\eqref{wtilde3} respectively, the solution of the gain kernel and the observer gain are given by
\begin{align}
\label{p1gain}P(x,y) =& \frac{\lambda}{\alpha} y\frac{I_1\left(\sqrt{\frac{\lambda}{\alpha} (y^2-x^2)}\right)}{\sqrt{\frac{\lambda}{\alpha} (y^2-x^2)}}, \\
\label{p3gain}P_1(x,s(t))=&-\alpha P(x,s(t)), 
\end{align}
where $I_1(x)$ is a modified Bessel function of the first kind. In order to show the negativity of estimation error, we recall a following lemma. The importance of such a property is stated in Section \ref{sec:constraints}. 
\begin{lem}\label{wmaximum}
Suppose that $\tilde{w}(0,t)<0$. Then, the solution of \eqref{wtilde1}-\eqref{wtilde3} satisfies 
\begin{align}
\tilde{w}(x,t)<0, \quad \forall x\in(0,s(t)), \quad \forall t>0.
\end{align}
\end{lem}
The proof of Lemma \ref{wmaximum} is given by maximum principle\cite{nonlinearPDE}.  
\subsubsection{Inverse transformation }
By the same procedure, one can derive the inverse transformation from $\tilde{w}$-system in \eqref{wtilde1}-\eqref{wtilde3} to $\tilde{u}$-system in \eqref{eq:errorPDE}-\eqref{eq:errorBC2} as
\begin{align}\label{inverseerror}
\tilde{w}(x,t) =& \tilde{u}(x,t) - \int_{x}^{s(t)} Q(x,y) \tilde{u}(y,t) dy, \\
Q(x,y) =& \frac{\lambda}{\alpha} y\frac{J_1\left(\sqrt{\frac{\lambda}{\alpha} (y^2-x^2)}\right)}{\sqrt{\frac{\lambda}{\alpha} (y^2-x^2)}}, \label{bessel}
\end{align}
where $J_1(x)$ is a Bessel function of the first kind. 
\subsection{Output Feedback Control}
By equivalence, the transformation of the variables $(\hat{u},X)$ into $(\hat{w},X)$ leads to  the gain kernel functions defined by the state-feedback backstepping transformation \eqref{eq:DBST} given by
\begin{align}\label{eq:observerDBST}
\hat{w}(x,t)=\hat{u}(x,t)-\frac{c}{\alpha} \int_{x}^{s(t)} (x-y)\hat{u}(y,t) dy\nonumber\\
+\frac{c}{\beta}(s(t)-x) X(t),
\end{align}
with an associated  target system given by
\begin{align}\label{obsvtarPDE}
\hat{w}_t(x,t)=&\alpha \hat{w}_{xx}(x,t)+\frac{c}{\beta}\dot{s}(t)X(t)\nonumber\\
 &+ f(x,s(t))\tilde{w}_{x}(s(t),t),\\
\hat{w}(s(t),t)=&0, \quad \hat{w}_{x}(0,t)=0,\\
\dot{X}(t)=&-cX(t)-\beta \hat{w}_{x}(s(t),t) - \beta \tilde{w}_{x}(s(t),t) \label{obsvtarODE},
\end{align}
where 
\begin{align}
f(x,s(t)) =& P_1(x,s(t))- \frac{c}{\alpha}  \int_{x}^{s(t)} (x-y) P_1(y,s(t)) dy  \nonumber\\
&-c  (s(t)-x).
\end{align}
Evaluating the spatial derivative of \eqref{eq:observerDBST} at  $x=0$, we derive  the output feedback controller  as
\begin{align}\label{output}
q_c(t)&=-ck\left( \frac{1}{\alpha}\int_{0}^{s(t)} \hat{u}(x,t) dx+\frac{1}{\beta} X(t)\right). 
\end{align}
By the same procedure as \cite{Shumon16}, one can derive an inverse transformation written as 
\begin{align}\label{whatinv}
\hat{u}(x,t)=&\hat{w}(x,t)+ \frac{\beta}{\alpha}\int_{x}^{s(t)} \psi(x-y)\hat{w}(y,t) dy\notag\\
&+ \psi(x-s(t)) X(t), \\
\psi(x) =&\frac{c}{\beta} \sqrt{\frac{\alpha}{c}} {\rm sin} \left(\sqrt{\frac{c}{\alpha}} x \right). 
\end{align}

%Then, taking time derivative, we have
%\begin{align}
%\dot{q}_c(t)&=-cq_c(t)\nonumber\\
%&-ck\left( \frac{P_3(t)}{\beta}-\frac{c_1(t)}{2\alpha}s(t)^2-\frac{c_1(t)}{c} \right)\tilde{u}_x(s(t),t)
%\end{align}

\section{Physical Constraints}\label{sec:constraints}
The two physical constraints defined in \eqref{physical constraints} are required to guarantee the physical validity of the model  \eqref{eq:stefanPDE}-\eqref{eq:stefanODE} and achieve the control objective $s(t) \to s_r$. In this section, we derived sufficient conditions to  guarantee \eqref{physical constraints}  which is satisfied if only if  the controller \eqref{output} is always injecting positive heat without any interface overshoot beyond the setpoint. 
First, we state the following lemma. 
\begin{lem}\label{maximum}
Assume that the upper bound of the initial temperature is known, i.e. the Lipschitz constant $H$ in \eqref{eq:stefanICbound} is known. Suppose that we set the initial estimation $\hat{u}(x,0)$ and the gain parameter of the observer $\lambda $ chosen to satisfy
\begin{align}\label{cond1}
\hat{u}(x,0) =& \hat{H}(s_0-x), \\
\label{cond2}\lambda <&\frac{4\alpha}{s_0^2} \left(1-\frac{H}{\hat{H}}\right), 
\end{align}
respectively with a choice of a parameter $\hat{H}>H$. Then, the following properties hold:
\begin{align}\label{positive}
 &\tilde{u}(x,t) < 0, \quad \tilde{u}_{x}(s(t),t) > 0, \quad \forall x\in(0,s(t)), \quad \forall t>0
\end{align}
\end{lem}
\begin{proof}
Lemma \ref{wmaximum} showed that if $\tilde{w}(x,0)<0$, then $\tilde{w}(x,t)<0$. In addition, by the direct transformation \eqref{eq:errorDBST}, $\tilde{w}(x,t)<0$ leads to $\tilde{u}(x,t)<0$ due to the positivity of the solution to the gain kernel \eqref{p1gain}. Therefore, with the help of \eqref{inverseerror}, we deduce that $\tilde{u}(x,t)<0$ if $\tilde{u}(x,0)$ satisfies 
\begin{align}\label{condinv}
 \tilde{u}(x,0) < \int_{x}^{s_0} Q(x,y) \tilde{u}(y,0) dy, \quad \forall x\in(0,s_0). 
\end{align}
Considering the bound of the solution \eqref{bessel} under the condition of \eqref{cond1}, the sufficient condition for \eqref{condinv} to hold is given by \eqref{cond2} which restricts the gain parameter $\lambda$. Thus, we have shown that conditions \eqref{cond1} and \eqref{cond2} lead to $\tilde{u}(x,t)<0$ for all $ x\in(0,s_0)$. In addition, by the boundary condition \eqref{eq:errorBC2} and Hopf's lemma, it leads to $\tilde{u}_{x}(s(t),t)>0$. 
\end{proof}

Next, we show the physical constraints \eqref{physical constraints} are satisfied with a restriction on the setpoint using \eqref{positive}. 
\begin{prop}\label{proposition}
Suppose the initial values $(\hat{u}_0(x), s_0)$ satisfy \eqref{cond1} and the setpoint $s_r$ is chosen to satisfy 
\begin{align}\label{allcondition}
&s_r > s_0+\frac{\beta s_0^2}{2\alpha } \hat{H}.
\end{align}
Then, the following physical constraints are satisfied 
\begin{align}\label{constraint1}
q_c(t)>&0, \quad u(x,t)>0, \quad \dot{s}(t)>0, \\
\label{constraint2}s_0<&s(t)<s_r.
\end{align}
\end{prop}
\begin{proof}
Taking the time derivative of \eqref{output} along with the solution \eqref{obsfirst}--\eqref{obslast}, with the help of the observer gain \eqref{p3gain},  we derive the following differential equation:
\begin{align}\label{odecontrol}  
&\dot{q}_c(t)=-cq_c(t)+\left( 1+ \int_0^{s(t)} P(x,s(t)) dx \right)\tilde{u}_x(s(t),t).
\end{align}
From the positivity of the solution \eqref{p1gain} and the Neumann boundary value \eqref{positive} by Lemma \ref{maximum}, it leads to the following differential inequality
\begin{align}
&\dot{q}_c(t)\geq -cq_c(t).
\end{align}
Hence, if the initial values satisfy $q_{c}(0)>0$, equivalently \eqref{allcondition} by  \eqref{output} and \eqref{cond1}, we get
\begin{align}\label{positivecon}
 q_{c}(t)>0, \quad \forall t>0. 
 \end{align}
By Lemma \ref{lemma1}, we derive the two other conditions stated in \eqref{constraint1}. Then, with the relation \eqref{positive} given in Lemma \ref{maximum} and the positivity of $u(x,t)$ derived in \eqref{constraint1}, the following inequality is established  
\begin{align}\label{positivehat}
\hat{u}(x,t) >0, \quad \forall x \in(0,s(t)), \quad  \forall t>0.
\end{align}
Finally, substituting the inequalities \eqref{positivecon} and \eqref{positivehat} into \eqref{output}, we arrive at
\begin{align}
X(t) <0, \quad \forall t>0
\end{align}
which guarantees the second physical constraint \eqref{constraint2}. 
\end{proof}
In the next section, we assume that the restrictions \eqref{cond1}, \eqref{cond2} and \eqref{allcondition} are satisfied. We use the physical constraints \eqref{constraint1} and \eqref{constraint2} to show the Lyapunov stability. 

\section{Lyapunov Stability}\label{stability}
In this section, we establish the exponential stability of the origin in the closed-loop system in ${\cal H}_1$-norm  based on the Lyapunov stability analysis for PDEs\cite{krstic2008boundary} of the associated target system \eqref{wtilde1}-\eqref{wtilde3} and \eqref{obsvtarPDE}-\eqref{obsvtarODE}. 
\subsection{Stability Analysis of $\tilde{w}$-System }
Let $\tilde{V}_1$ be a functional such that
\begin{align}\label{eq:lyapunov}
\tilde{V}_1 = &\frac{1}{2} ||\tilde{w}||_{{\cal H}_1}^2.
\end{align}
Taking the derivative of \eqref{eq:lyapunov} along the solution of the target system \eqref{wtilde1}-\eqref{wtilde3} , we obtain 
\begin{align}
\dot{\tilde{V}}_1 = -\alpha|| \tilde{w}_{x}||_{{\cal H}_{1}}^2  - \lambda  ||\tilde{w}||_{{\cal H}_1}^2  -\frac{\dot{s}(t)}{2}\tilde{w}_{x}(s(t),t)^2.
\end{align}
With the help of \eqref{constraint1} and \eqref{constraint2}, and applying Pointcare's inequality, a differential inequality in  $\tilde{V}_1$ is obtained as
\begin{align}
\dot{\tilde{V}}_1\leq - \left( \frac{\alpha}{4s_r^2} + \lambda\right) \tilde{V}_1.
\end{align}
Hence, the origin of the target $\tilde{w}$-system is exponentially stable. Since the transformation \eqref{eq:errorDBST} is invertible as in \eqref{inverseerror}, the exponential stability of $\tilde{w}$-system at the origin deduces the exponential stability of $\tilde{u}$-system at the origin with the help of \eqref{constraint2}, which completes the proof of Theorem \ref{observer}. 

\subsection{Stability Analysis of the Closed-Loop System}
Define the functional $V_{tot}$ as 
\begin{align}\label{eq:obsvlyapunov}
V_{tot} = &\frac{1}{2} ||\hat{w}||_{{\cal H}_1}^2+\frac{p}{2}X(t)^2+d\tilde{V}_1,
\end{align}
where $d$ is chosen to be large enough and $p$ is defined as
\begin{align}
&p = \frac{c \alpha}{16 \beta^2 s_r}.
\end{align}
Taking the derivative of \eqref{eq:lyapunov} along the solution of the target system \eqref{obsvtarPDE}-\eqref{obsvtarODE}, and applying Young's, Cauchy-Schwarz, Poincare's, Agmon's inequality, with the help of  \eqref{constraint1} and \eqref{constraint2}, the following  holds:
\begin{align}\label{estimatelyap}
\dot{V}_{tot}\leq & -bV_{tot} +a\dot{s}(t)V_{tot},
\end{align}
where,
\begin{align}
&a = {\rm max} \left\{ s_r^2, \frac{16cs_r}{\alpha}\right\}, \quad b = {\rm min}\left\{\frac{\alpha}{8s_r^2}, c, 2\lambda\right\} .
\end{align}
Define the Lyapunov function $V$ such that 
\begin{align}\label{finlyap}
V = V_{tot} e^{-a s(t)}. 
\end{align}
Taking the time derivative of \eqref{finlyap} and applying \eqref{estimatelyap}, we get
\begin{align}
\dot{V} = \left(\dot{V}_{tot} -a\dot{s}(t)V_{tot}\right)e^{-bt} \leq -bV
\end{align}
From the equation above, we deduce the exponential decay of $V$ as $V(t) \leq V(0) e^{-bt}$. Then, using \eqref{constraint2}, we arrive at 
\begin{align}
V_{tot} \leq e^{as_r} V_{tot}(0)  e^{-bt}.
\end{align}
Hence, the origin of the $(\hat{w}, X, \tilde{w})$-system is exponentially stable. Since the transformation \eqref{eq:errorDBST} and \eqref{eq:observerDBST} are  invertible as described in \eqref{inverseerror} and \eqref{whatinv},  the exponential stability of $(\hat{w}, \hat{X}, \tilde{w})$-system at the origin guarantees the exponential stability of $(\hat{u}, \hat{X}, \tilde{u})$-system at the origin, which completes the proof of Theorem \ref{outputthm}. 

\section{Simulation Results }\label{simulation}

As in \cite{Shumon16}, the simulation is performed considering a strip of zinc whose physical properties are given in Table 1. The setpoint and the initial values are chosen as $s_{r}$ = 0.35 m, $s_0$ = 0.01 m, $T_0(x)-T_m=H(s_{0}-x)$ with $H$ = 100 K$\cdot {\rm m}^{-1}$, and $\hat{T}_0(x)-T_m=\hat{H}(s_{0}-x)$ with $\hat{H}$ = 1000 K$\cdot {\rm m}^{-1}$. The controller gain $c$ = 0.001 and the gain parameter of the observer $\lambda$ = 0.001 are chosen. Then, the restriction on $\hat{T}_0(x)$, $\lambda$, and $s_r$ described in \eqref{restriction1}-\eqref{restriction3} are satisfied, which are conditions for Theorem \ref{observer} and Theorem \ref{outputthm} to remain valid. \\

\begin{figure}[t]
\centering
\subfloat[The interface $s(t)$ converges to the setpoint $s_r$ keeping increasing and without the overshoot, i.e. $\dot{s}(t)>0$, $s_0<s(t)<s_r$.]{\includegraphics[width=3.0in]{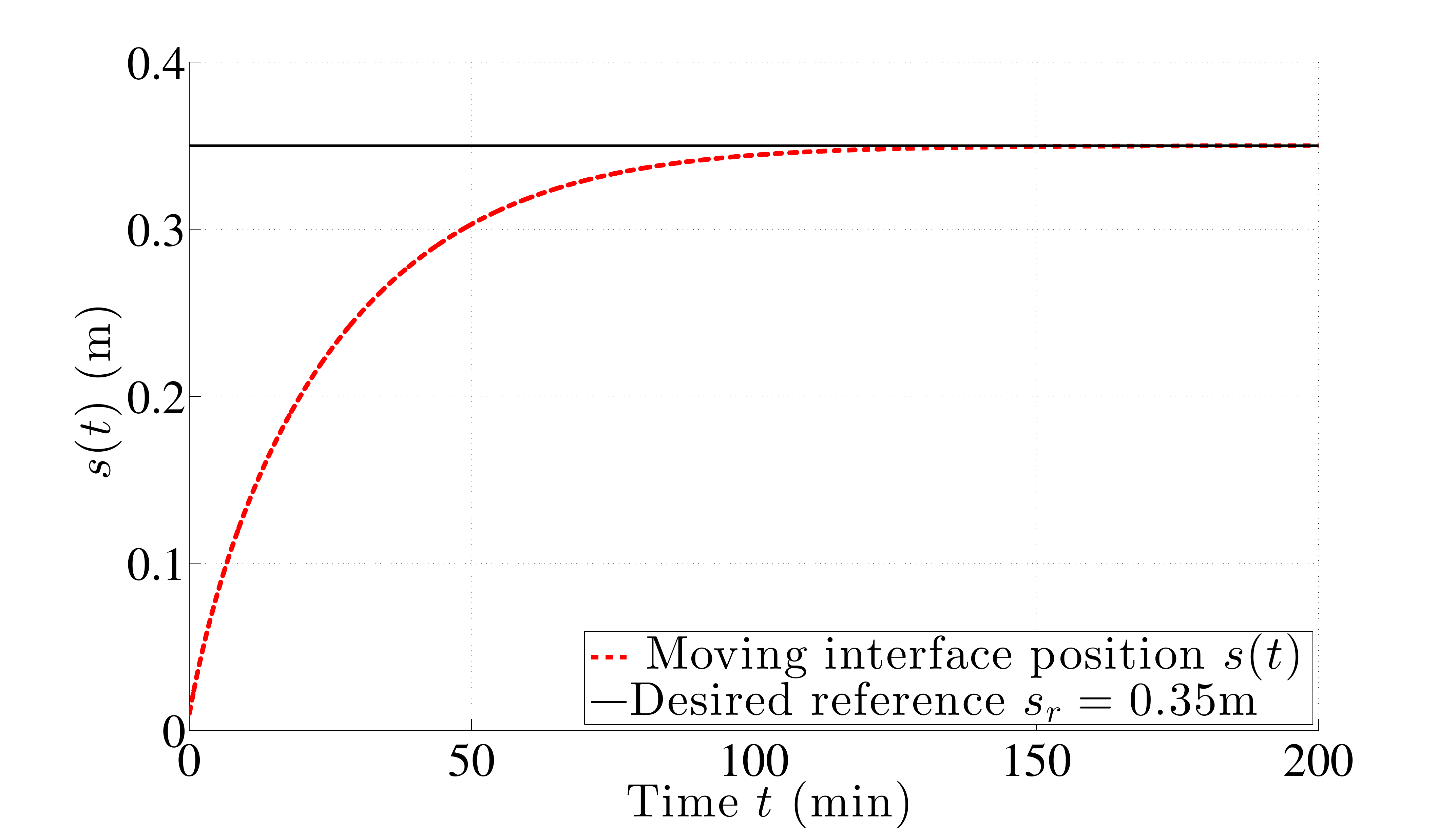}}\\
\subfloat[The output feedback controller $q_{c}(t)$ remains positive, i.e. $q_{c}(t)>0$. ]{\includegraphics[width=3.0in]{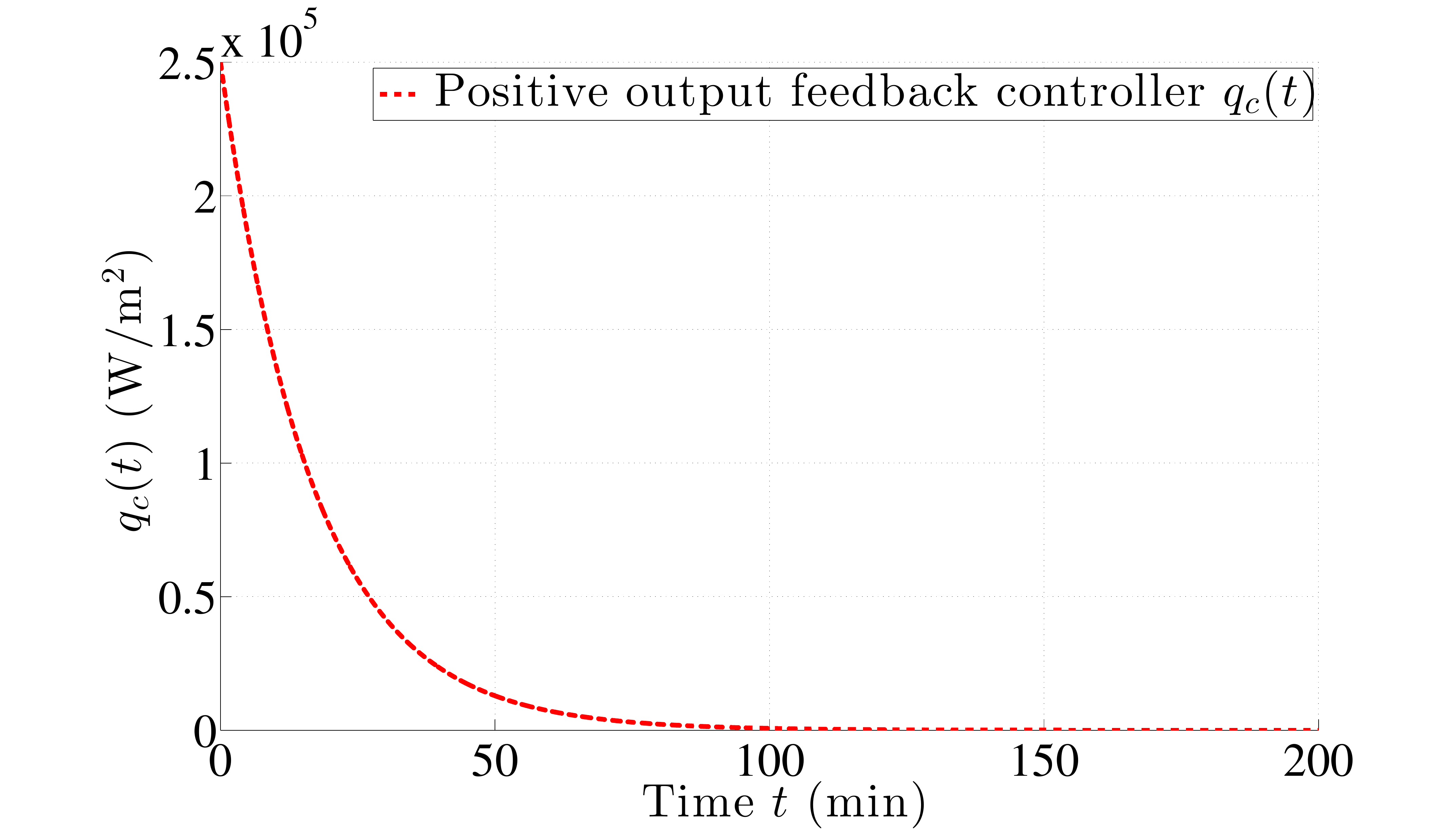}}\\
\subfloat[The estimation error of the boundary temperature $\tilde{T}(0,t)$ converges to zero with keeping the negativity, i.e. $\tilde{T}(0,t)<0$. ]{\includegraphics[width=3.0in]{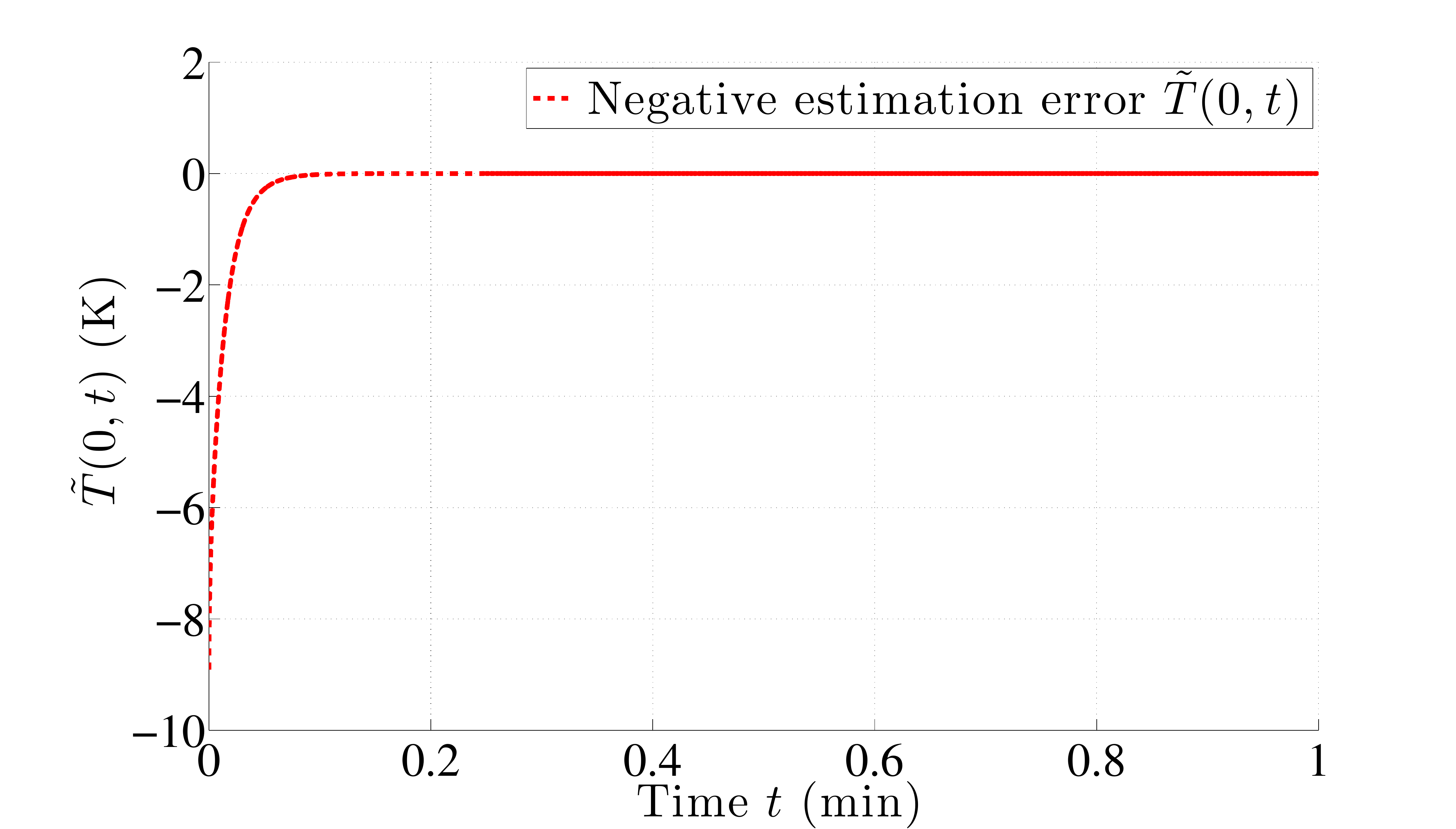}}
\caption{The simulation of the closed-loop system \eqref{eq:stefanPDE} - \eqref{eq:stefanODE} and the estimator  \eqref{observerPDE} - \eqref{P1x} with the output feedback control law \eqref{outctr}.}
\label{fig:simulation}
\end{figure}

The dynamics of the moving interface $s(t)$, the output feedback controller $q_{c}(t)$, and the estimation error of the boundary temperature $\tilde{T}(0,t)$ are depicted in Fig. \ref{fig:simulation} (a) - (c), respectively. Fig. \ref{fig:simulation} (a) shows that the interface $s(t)$ converges to the setpoint $s_r$ with $\dot{s}(t)>0$ and $s_0<s(t)<s_r$ for $\forall t>0$, which are guaranteed in Proposition \ref{proposition}.  Fig. \ref{fig:simulation} (b) shows that the output feedback controller remains positive, which is a physical constraint for the model to be valid as stated in Lemma \ref{lemma1} and ensured in Proposition \ref{proposition}. The positivity of the controller results from the negativity of the distributed estimation error $\tilde{T}(x,t)$ as shown in Lemma \ref{maximum} and Proposition \ref{proposition}. Fig. \ref{fig:simulation} (c) shows that the estimation error of boundary temperature $\tilde{T}(0,t)$ converges to zero and remains negative. Therefore, the numerical results are consistent with the theoretical results stated in Lemma \ref{maximum} and Proposition \ref{proposition}. 

\begin{table}[htb]
\caption{Physical properties of zinc}
\begin{center}
    \begin{tabular}{| l | l | l | }
    \hline
    $\textbf{Description}$ & $\textbf{Symbol}$ & $\textbf{Value}$ \\ \hline
    Density & $\rho$ & 6570 ${\rm kg}\cdot {\rm m}^{-3}$\\ 
    Latent heat of fusion & $\Delta H^*$ & 111,961${\rm J}\cdot {\rm kg}^{-1}$ \\ 
    Heat Capacity & $C_p$ & 389.5687 ${\rm J} \cdot {\rm kg}^{-1}\cdot {\rm K}^{-1}$  \\  
    Thermal conductivity & $k$ & 116 ${\rm W}\cdot {\rm m}^{-1}$  \\ \hline
    \end{tabular}
\end{center}
\end{table}

\section{Conclusions and Future Works}\label{conclusion}
In this paper we designed  an observer  and a boundary output feedback controller for the one-phase Stefan problem via the backstepping transformation. The proposed controller achieves the exponential stability of the closed-loop system using only  a measurement of the moving interface and ensures that the  physical constraints are satisfied under the restriction on the setpoint and the gain parameter of the designed observer assuming the upper bound of the initial temperature to be known. The main contribution of this paper is that, this is the first result which shows the convergence of the estimation error and output feedback control of the one-phase Stefan problem theoretically. Although the Stefan problem has been a well known model since 200 years ago related with phase transitions which appear in various nature and engineering processes, its control and  estimation related problems have not been investigated in details. The estimation of the sea-ice melting problem in Arctic region is being considered as a future work. 

\bibliographystyle{unsrt}
\bibliography{ref}

\begin{thebibliography}{10}

\bibitem{wett91}
J.S. Wettlaufer.
\newblock Heat flux at the ice-ocean interface.
\newblock {\em Journal of Geophysical Research}, 96(C4):297--313, 1991.

\bibitem{petrus2012}
B.~Petrus, J.~Bentsman, and B.G. Thomas.
\newblock Enthalpy-based feedback control algorithms for the stefan problem.
\newblock In {\em CDC}, pages 7037--7042, 2012.

\bibitem{conrad_90}
F.~Conrad, D.~Hilhorst, and T.~I. Seidman.
\newblock Well-posedness of a moving boundary problem arising in a
  dissolution-growth process.
\newblock {\em Nonlinear Analysis}, 15(5):445 -- 465, 1990.

\bibitem{Belen03}
B.~Zalba, J.M. Marin, L.F. Cabeza, and H.~Mehling.
\newblock Review on thermal energy storage with phase change: materials, heat
  transfer analysis and applications.
\newblock {\em Applied Thermal Engineering}, 23(3):251 -- 283, 2003.

\bibitem{Gupta03}
S.~Gupta.
\newblock {\em The classical Stefan problem. Basic concepts, Modelling and
  Analysis}.
\newblock Applied mathematics and Mechanics. North-Holland, 2003.

\bibitem{maidi2014}
A.~Maidi and J.-P. Corriou.
\newblock Boundary geometric control of a linear stefan problem.
\newblock {\em Journal of Process Control}, 24(6):939--946, 2014.

\bibitem{krstic2008boundary}
M.~Krstic and A.~Smyshlyaev.
\newblock {\em Boundary control of PDEs: A course on backstepping designs},
  volume~16.
\newblock Siam, 2008.

\bibitem{andrew2004}
A.~Smyshlyaev and M.~Krstic.
\newblock Closed-form boundary state feedbacks for a class of 1-d partial
  integro-differential equations.
\newblock {\em Automatic Control, IEEE Transactions on}, 49(12):2185--2202, Dec
  2004.

\bibitem{Shumon16}
S.~Koga, M.~Diagne, S.~Tang, and M.~Krstic.
\newblock Backstepping control of the one-phase stefan problem.
\newblock In {\em 2016 American Control Conference (ACC)}, pages 2548--2553.
  IEEE, 2016.

\bibitem{Izadi15}
M.~Izadi and S.~Dubljevic.
\newblock Backstepping output-feedback control of moving boundary parabolic
  {PDE}s.
\newblock {\em European Journal of Control}, 21(0):27 -- 35, 2015.

\bibitem{krstic2009}
M.~Krstic.
\newblock Compensating actuator and sensor dynamics governed by diffusion
  {PDE}s.
\newblock {\em Systems \& Control Letters}, 58(5):372--377, 2009.

\bibitem{Gian2011}
G.A. Susto and M.~Krstic.
\newblock Control of {PDE}--{ODE} cascades with neumann interconnections.
\newblock {\em Journal of the Franklin Institute}, 347(1):284--314, 2010.

\bibitem{nonlinearPDE}
C.V. Pao.
\newblock {\em Nonlinear Parabolic and Elliptic Equations}.
\newblock Springer, 1992.

\end{thebibliography}

\end{document}